\documentclass[a4paper,reqno,11pt]{amsart}
\usepackage{amsmath}
\usepackage{amssymb}
\usepackage{cases}
\usepackage{enumitem}
\allowdisplaybreaks[4]

\newtheorem{thm}{Theorem}[section]

\newtheorem{lem}[thm]{Lemma}

\newtheorem{conj}[thm]{Conjecture}
\newtheorem{claim}{Claim}

\theoremstyle{definition}
\newtheorem{definition}[thm]{Definition}

\theoremstyle{remark}

\numberwithin{equation}{section}

\newcommand{\bQ}{\mathbb{Q}}
\newcommand{\bP}{\mathbb{P}}
\newcommand{\bF}{\mathbb{F}}

\newcommand\OO{{\mathcal{O}}}
\newcommand\EE{{\mathcal{E}}}

\newcommand\Vol{\text{\rm Vol}}

\newcommand\FF{{\mathcal{F}}}

 \newcommand\mT{{\mathcal{T}}}

\newcommand\lcm{{\text{l.c.m.}}}

\newcommand\mult{{\rm{mult}}}
\newcommand\Nklt{{\rm{Nklt}}}
\newcommand\Nlc{{\rm{Nlc}}}

\newcommand\lct{{\rm{lct}}}
\newcommand\glct{{\rm{glct}}}

\newcommand\Ne{{N_\epsilon}}

\newcommand\Ve{{V_\epsilon}}

\begin{document}

\title{Boundedness of threefolds of Fano type with Mori fibration structures}

\author{Chen Jiang}

\address{Shanghai Center for Mathematical Sciences, Fudan University, Jiangwan Campus, 2005 Songhu Road, Shanghai, 200438, China}
\email{chenjiang@fudan.edu.cn}

\thanks{The author was supported by Grant-in-Aid for JSPS Fellows (KAKENHI No. 25-6549).}

\begin{abstract}
We show boundedness of $3$-folds of $\epsilon$-Fano type with Mori fibration structures. The proof is based on the birational boundedness result in our previous work \cite{Jiang6} combining with arguments in Kawamata \cite{K} and Koll\'ar--Miyaoka--Mori--Takagi \cite{KMMT}.
\end{abstract} 

\maketitle
\pagestyle{myheadings} \markboth{\hfill  C. Jiang
\hfill}{\hfill Boundedness of $3$-folds of Fano type with Mori fibration structures\hfill}

\tableofcontents

\section{Introduction}
Throughout this paper,  we work over the field of complex numbers $\mathbb{C}$.
See Subsection \ref{section notation} for notation and conventions.

A normal projective variety $X$ is \emph{of $\epsilon$-Fano type} if there exists an effective $\bQ$-divisor $B$ such that $(X, B)$ is an $\epsilon$-klt log Fano pair. $0$-Fano type is also called Fano type for simplicity. The notion of Fano type was introduced by Prokhorov--Shokurov \cite{PS}.

We are mainly interested in the boundedness of varieties of $\epsilon$-Fano type. 
Our motivation  is the following conjecture due to A. Borisov, L. Borisov, and V. Alexeev.

\begin{conj}[{BAB} Conjecture]
Fix an integer $n>0$, $0<\epsilon<1$.
Then the set of all
$n$-dimensional varieties of $\epsilon$-Fano type  is bounded.
\end{conj}
 {BAB} Conjecture is one of the most important conjectures in birational geometry. It is related to the termination of flips (cf. \cite{B, BS}) and has interesting application for the Jordan property of Cremona groups (cf. \cite{ProkhorovShramov}). Besides, varieties of Fano type form a fundamental class in birational geometry according to Minimal Model Program and have many interesting properties (cf. \cite{CG, GOST, Krylov}). Hence it is very interesting to understand the basic properties of this class, such as boundedness.

{BAB} Conjecture in dimension two was proved  by Alexeev \cite{AK2} with a simplified argument by Alexeev--Mori \cite{AM}. In higher dimension, BAB Conjecture  still remains open. There are only  some partial boundedness results (cf.  \cite{BB, KMM92, K, KMMT, AB}). 

We recall the following theorem proved in \cite{Jiang} by using Minimal Model Program.
\begin{thm}[{cf. \cite[Proof of Theorem 2.3]{Jiang}}]\label{thm MFS}Fix an integer $n>0$ and $0<\epsilon<1$. 
Every $n$-dimensional variety $X$ of $\epsilon$-Fano type is birational to an $n$-dimensional variety  $X'$ of $\epsilon$-Fano type  with a Mori fibration structure.
\end{thm}

We will recall the proof in Section \ref{section proof 1}. Here we emphasize that having a Mori fibration structure implies that having at most $\bQ$-factorial terminal singularities by our definition (see Subsection \ref{notation}). According to this theorem, it is important and interesting to investigate varieties of $\epsilon$-Fano type with Mori fibration structures. In fact, in the proof of BAB Conjecture in dimension two (cf. \cite{AK2, AM}), the first step is to classify (and bound) all surfaces  of $\epsilon$-Fano type with Mori fibration structures, which are just projective plane or Hirzebruch surfaces $\bF_n$ with $n<2/\epsilon$. Therefore, we are interested in the boundedness of  $3$-folds of $\epsilon$-Fano type with Mori fibration structures, as the first step towards BAB Conjecture in dimension three.

The following is our main theorem.
\begin{thm}\label{main}
Fix $0<\epsilon<1$. The set of all
$3$-folds of  $\epsilon$-Fano type with Mori fibration structures is  bounded.
\end{thm}

For interesting examples in dimension three, we refer to \cite{Lin}, where $3$-folds of Fano type with conic bundle structures were constructed, which proves the birational unboundedness of   $3$-folds of $0$-Fano type. 
\subsection{Sketch of the proof}
Let $X$ be a $3$-fold of $\epsilon$-Fano type with a Mori fibration $f:X\to Z$. If $\dim Z=0$, then $X$ is a $\bQ$-factorial terminal Fano $3$-fold of Picard number one, which is bounded by Kawamata \cite{K}. So we only need to consider the case when $\dim Z>0$. 

We recall the following theorem from \cite{Jiang6}, by which we proved the birational boundedness of $3$-folds of $\epsilon$-Fano type.
\begin{thm}[{cf. \cite[Proof of Corollaries 1.5, 1.8]{Jiang6}}]\label{NV}
Fix $0<\epsilon<1$. Then there exist positive integers $N_{\epsilon}$ and $V_{\epsilon}$ depending only on $\epsilon$, with the following property:

If $X$ is a $3$-fold of $\epsilon$-Fano type with a Mori fibration $f:X\to Z$.
\begin{itemize}
\item[(1)] If $\dim Z=1$ (i.e. $Z=\bP^1$), take a general fiber $F$ of $f$, then

\begin{itemize}
    \item[(1-1)]   $-K_X+N_\epsilon F$ is ample;

    \item[(1-2)] $(-K_X+N_\epsilon F)^3\leq V_\epsilon$.
  \end{itemize}
\item[(2)] If $\dim Z=2$,  then there exists a very ample divisor $H$ on $Z$ such that

\begin{itemize}
    \item[(2-1)]   $-K_X+N_\epsilon f^*H$ is ample;

    \item[(2-2)] $(-K_X+N_\epsilon  f^*H)^3\leq V_\epsilon$.
  \end{itemize}
\end{itemize}
\end{thm}
We will recall the proof in Section \ref{section proof 2}. By this theorem, to show the boundedness, it suffices to show the boundedness of Gorenstein indices due to Koll\'ar's effective base point free theorem (see Subsection \ref{subsection proof main}).
We will explain the idea of bounding the Gorenstein indices.

For convenience, we define $G$ and $\FF_X$ as following.

\begin{definition}\label{GF}
Let $X$ be a $3$-fold of $\epsilon$-Fano type with a Mori fibration $f:X\to Z$ such that $\dim Z>0$. 
Keep the notation in Theorem \ref{NV}.
\begin{enumerate}
 \item Define the projective smooth surface $G$ to be a general fiber  $F$ (resp. a general element of $|f^*(H)|$) if $\dim Z=1$ (resp. $\dim Z=2$).
\item Define the torsion free sheaf $\FF_X:=\mT_X\oplus \OO_X(N_\epsilon G)$, where $\mT_X$ is the tangent sheaf on $X$.
\end{enumerate}
\end{definition}
We remark that $G$ is a del Pezzo surface (resp. conic bundle over a general $H$) if $\dim Z=1$ (resp. $\dim Z=2$).

Following the idea of Koll\'ar--Miyaoka--Mori--Takagi \cite{KMMT}, we can prove the pseudo-effectivity of $c_2(\FF_X)$.
\begin{thm}\label{F pseff}
Fix $0<\epsilon<1$. Let $X$ be a $3$-fold of $\epsilon$-Fano type with a Mori fibration $f:X\to Z$ such that $\dim Z>0$. 
Keep the notation in Definition \ref{GF}.  Then
$c_2(\FF_X)$ is pseudo-effective. 
\end{thm}

Following the idea of Kawamata \cite{K}, after bounding $(-K_X)\cdot c_2(X)$ from below, we can get an upper bound for Cartier index of $K_X$, which eventually implies the desired boundedness.
\begin{thm}\label{bound of r}
Fix $0<\epsilon<1$. Let $X$ be a $3$-fold of $\epsilon$-Fano type with a Mori fibration $f:X\to Z$ such that $\dim Z>0$, then
\begin{enumerate}
\item $(-K_X)\cdot c_2(X)\geq -M_\epsilon$;
\item $r_\epsilon K_X$ is Cartier, 
\end{enumerate}
where $M_\epsilon=5\Ve+12\Ne$ and $r_\epsilon=(24+M_\epsilon)!$, $\Ne$ and $\Ve$ are the numbers defined in Theorem \ref{NV}.
\end{thm}

This paper is organized as follows. For the reader's convenience, in Sections \ref{section proof 1} and \ref{section proof 2}, we recall the proof of Theorems \ref{thm MFS} and \ref{NV}. In Section \ref{section proof main}, we prove our main theorems, Theorems \ref{F pseff}, \ref{bound of r}, and \ref{main}.

\section{Preliminaries}
\subsection{Notation and conventions}\label{section notation}\label{notation}
 We adopt the standard notation and definitions in \cite{KMM} and \cite{KM}, and will freely use them.

A {\it pair} $(X, B)$ consists of  a normal projective variety $X$ and an effective
$\mathbb{Q}$-divisor $B$ on $X$ such that
$K_X+B$ is $\mathbb{Q}$-Cartier.   

The pair $(X, B)$ is called a
\emph{log Fano pair} if $-(K_X+B)$ is ample.

Let $f: Y\rightarrow X$ be a log
resolution of the pair $(X, B)$, write
$$
K_Y =f^*(K_X+B)+\sum a_iF_i,
$$
where $\{F_i\}$ are distinct prime divisors.  For some $\epsilon \in [0,1]$, the
pair $(X,B)$ is called
\begin{itemize}
\item[(a)] \emph{$\epsilon$-kawamata log terminal} (\emph{$\epsilon$-klt},
for short) if $a_i> -1+\epsilon$ for all $i$;

\item[(b)] \emph{$\epsilon$-log canonical} (\emph{$\epsilon$-lc}, for
short) if $a_i\geq  -1+\epsilon$ for all $i$;

\item[(c)] \emph{terminal} if  $a_i> 0$ for all $f$-exceptional divisors $F_i$ and all $f$. 
\end{itemize}
Usually we write $X$ instead of $(X,0)$ in the case $B=0$.
Note that $0$-klt (resp. $0$-lc) is just klt (resp. lc) in the usual sense.

  $F_i$ is called a {\it non-klt place}  (resp. {\it non-lc place}) of $(X, B)$  if $a_i\leq -1$ (resp. $<-1$).
A subvariety $V\subset X$ is called a {\it non-klt center} (resp. {\it non-lc center}) of $(X, B)$ if it is the image of a non-klt place (resp. non-lc place). The {\it non-klt locus} $\text{Nklt}(X, B)$ is the union of all  non-klt centers of $(X, B)$. The {\it non-lc locus} $\text{Nlc}(X, B)$ is the union of all  non-lc centers of $(X, B)$.

A normal projective variety $X$ is \emph{of $\epsilon$-Fano type} if there exists an effective $\bQ$-divisor $B$ such that $(X, B)$ is an $\epsilon$-klt log Fano pair. $0$-Fano type is also called Fano type for simplicity.

A projective morphism $f:X\to Z$ between normal projective varieties is called a {\it Mori fibration} (or {\it Mori fiber space}) if
\begin{enumerate}
\item $X$ is $\bQ$-factorial with terminal singularities; 

\item $f$ is a {\it contraction}, i.e., $f_*\OO_X=\OO_Z$;

\item $-K_X$ is ample over $Z$;
\item $\rho(X/Z)=1$;
\item $\dim X > \dim Z$. 
\end{enumerate}
We say that $X$ is with a {\it Mori fibration structure} if there exists a Mori fibration $X\to Z$. In particular, in this situation, $X$ has at most $\bQ$-factorial terminal singularities by definition.

A collection of varieties $\{X_t\}_{t\in T}$ is
said to be \emph{bounded} (resp.  \emph{birationally bounded}) if there exists $h:\mathcal{X}\rightarrow
S$ a projective morphism between schemes of finite type such that 
each $X_t$ is isomorphic (resp. birational) to $\mathcal{X}_s$ for some $s\in S$.

\subsection{Volumes}

Let $X$ be an $n$-dimensional projective variety  and $D$ be a Cartier divisor on $X$. The {\it volume} of $D$ is the real number
$$
{\Vol}(X, D)=\limsup_{m\rightarrow \infty}\frac{h^0(X,\OO_X(mD))}{m^n/n!}.
$$
Note that the limsup is actually a limit. Moreover by the homogenous property of  volumes, we can extend the definition to $\bQ$-Cartier $\bQ$-divisors. Note that if $D$ is a nef $\bQ$-divisor, then $\Vol(X, D)=D^n$.

For more background on volumes, see \cite[11.4.A]{Positivity2}. It is easy to see the following inequality for volumes by comparing global sections by exact sequences.
\begin{lem}[{\cite[Lemma 2.5]{Jiang6}}]\label{lemma volume}
Let $X$ be a projective normal variety, $D$ a $\bQ$-Cartier $\bQ$-divisor and $S$ a semi-ample normal Cartier prime divisor. Then for any rational number $q>0$,
$$
\Vol(X,D+qS)\leq \Vol(X, D)+q(\dim X) \Vol(S, D|_S+qS|_S).
$$
\end{lem}

\subsection{Log canonical thresholds and $\alpha$-invariants}
Let $(X, B)$ be  an lc pair and $D\geq 0$ be a $\bQ$-Cartier $\bQ$-divisor. The
{\it log canonical threshold} of $D$ with respect to $(X, B)$ is
$$\lct(X, B; D) = \sup\{t\in \bQ \mid (X, B+ tD) \text{ is lc}\}.$$
For application, we need to consider the case when $D$ is not effective. 
Let $G$ be a $\bQ$-Cartier $\bQ$-divisor satisfying $G+B\geq 0$, The
{\it generalized log canonical threshold} of $G$ with respect to $(X, B)$ is
$$\glct(X, B; G) = \sup\{t\in [0,1] \cap \bQ \mid (X, B+ tG) \text{ is lc}\}.$$
Note that the assumption $t\in [0,1]$ guarantees that $B+ tG\geq 0$.

If  $(X, B)$ is an lc log Fano pair,  the {\it  (generalized)  $\alpha$-invariant} of $(X, B)$ is defined by 
$$
\alpha(X, B)=\inf\{\glct(X,B;G)\mid G\sim_\bQ-(K_X+B), G+B\geq 0\}.
$$

\subsection{Non-klt centers}

The following lemma suggests a standard way to construct non-klt centers.
\begin{lem}[{cf. \cite[Lemma 2.29]{KM}}]\label{dimk}
Let $(X, B)$ be a pair and $V\subset X$ a closed subvariety of codimesion $k$ such that $V$ is not contained in the singular locus of $X$. If $\mult_V B\geq k$, then $V$ is a non-klt center of $(X, B)$. 
\end{lem}
Recall that the {\it multiplicity} $\mult_VF$ of a divisor $F$ along a subvariety $V$ is defined by the multiplicity $\mult_xF$ of $F$ at a general point $x\in V$.

Unfortunately, the converse of Lemma \ref{dimk} is not true unless $k=1$. Usually we do not have good estimates for the multiplicity along a non-klt center but the following lemma.

\begin{lem}[{cf. \cite[Theorem 9.5.13]{Positivity2}}]
Let $(X, B)$ be a pair and $V\subset X$  a non-klt center of $(X, B)$ such that $V$ is not contained in the singular locus of $X$. Then $\mult_V B\geq 1$. 
\end{lem}

We have the following connectedness lemma of Koll\'{a}r and Shokurov for non-klt locus (cf.  
Shokurov \cite{Shokurov}, Koll\'{a}r \cite[17.4]{Kol92}).
\begin{thm}[Connectedness Lemma]
Let $f:X\rightarrow Z$ be a proper morphism of normal varieties with connected fibers and $D$ a $\bQ$-divisor such that $-(K_X+D)$ is $\bQ$-Cartier, $f$-nef, and $f$-big. Write $D=D^+-D^-$ where $D^+$ and $D^-$ are effective with no common components. If $D^-$ is $f$-exceptional (i.e. all of its components have image of codimension at least $2$), then ${\rm Nklt} (X,D)\cap f^{-1}(z)$ is connected for any $z\in Z$. 
\end{thm}

As an application, we have the following theorem on inversion of adjunction (cf. \cite[Theorem 5.50]{KM}). Here we only use a weak version.
\begin{thm}[Inversion of adjunction]
Let $(X, B)$ be a pair and $S\subset X$ a normal  Cartier divisor not contained in the support of $B$. Then $$\Nklt(X,B)\cap S\subset \Nklt(S, B|_S).$$ In particular, if $\Nklt(X, B)\cap S\neq \emptyset$, then $(S, B|_S)$ is not klt.
\end{thm}

\subsection{Length of extremal rays}
Recall the result on length of extremal rays due to Kawamata.
\begin{thm}[{\cite{Ka}}]\label{ext ray}
Let $(X, B)$ be a klt pair. Then every $(K_X+B)$-negative extremal ray $R$ is generated by a rational curve $C$ such that $$0<-(K_X+B)\cdot C\leq 2\dim X.$$
\end{thm}
However, we need to deal with non-klt pairs in application. We have a slightly generalization of this theorem for non-klt pairs.

\begin{thm}[{\cite[Theorem 1.1(5)]{Fujino}}]\label{ext ray2}
Let $(X, B)$ be a pair. Fix a $(K_X+B)$-negative extremal ray $R$. Assume that 
$$
R\cap \overline{NE}(X)_{\Nlc(X,B)} = \{0\},
$$
where 
$$
 \overline{NE}(X)_{\Nlc(X,B)} = {\rm Im}(\overline{NE}(\Nlc(X,B)) \to \overline{NE}(X)).
$$
Then $R$ is generated by a rational curve $C$ such that $$0<-(K_X+B)\cdot C\leq 2\dim X.$$
\end{thm}

\section{Proof of Theorem \ref{thm MFS}}\label{section proof 1}

In this section, for the reader's convenience, we recall the proof of Theorem \ref{thm MFS} from \cite{Jiang}. We start with two lemmas. The first lemma is about equivalent definitions of $\epsilon$-Fano type.

\begin{lem}[{cf. \cite[Lemma-Definition 2.6]{PS}}]\label{FT=CY}
Let $Y$ be a projective  normal variety, and $\epsilon\in [0,1)$. The following are equivalent:
\begin{enumerate}
\item $Y$ is of $\epsilon$-Fano type;

\item There exists an effective $\mathbb{Q}$-divisor  $\Delta$ such that $\Delta$ is big, $(Y, \Delta)$ is $\epsilon$-klt, and $K_Y+\Delta\equiv 0$.
\end{enumerate}
\end{lem}
\begin{proof}
First we assume that  $Y$ is of $\epsilon$-Fano type, that is, there exists an effective $\mathbb{Q}$-divisor $B$ on $Y$ such that $(Y, B)$ is $\epsilon$-klt log Fano pair. Then take a general effective ample $\mathbb{Q}$-divisor $A$ on $Y$ such that $(Y, B+A)$ is $\epsilon$-klt and 
$$
K_Y+B+A\sim_{\mathbb{Q}}0.
$$
We may take $\Delta=A+B$.

Then we assume that there exists an effective $\mathbb{Q}$-divisor  $\Delta$ such that $\Delta$ is big, $(Y, \Delta)$ is $\epsilon$-klt, and $K_Y+\Delta\equiv 0$. Since $\Delta$ is big, we may write $\Delta=A+G$ where $A$ is an ample $\mathbb{Q}$-divisor and $G$ is an effective $\mathbb{Q}$-divisor. We may take a sufficiently small $\delta>0$ such that $(Y, \Delta+\delta G)$ is again $\epsilon$-klt. Hence $(Y, (1-\delta)\Delta+\delta G)$ is $\epsilon$-klt, and 
$$-(K_Y+ (1-\delta)\Delta+\delta G)\equiv \delta A$$
is ample. Hence $Y$ is of $\epsilon$-Fano type.
\end{proof}

Being of $\epsilon$-Fano type is preserved by MMP according to the following lemma.
\begin{lem}[{cf. \cite[Lemma 3.1]{GOST}}]\label{mmp lem}
Let $Y$ be a projective  normal variety and $f: Y\rightarrow Z$ be a projective birational contraction.  
\begin{enumerate}
\item If $Y$ is of $\epsilon$-Fano type, so is $Z$;

\item Assume that $f$ is small, then $Y$ is of $\epsilon$-Fano type if and only if so is $Z$.
\end{enumerate}
In particular, minimal model program preserves $\epsilon$-Fano type.
\end{lem}

\begin{proof}
First we assume that  $Y$ is of $\epsilon$-Fano type, that is, by Lemma \ref{FT=CY}, there exists an effective $\mathbb{Q}$-divisor  $\Delta$ such that $\Delta$ is big, $(Y, \Delta)$ is $\epsilon$-klt, and $K_Y+\Delta\equiv 0$. 
Pushing forward by $f$, by negativity lemma, 
$$
K_Y+\Delta=f^*(K_Z+f_*\Delta)\equiv 0.
$$
Hence $f_*\Delta$ is big, $(Z,f_*\Delta)$ is $\epsilon$-klt, and $K_Z+f_*\Delta\equiv 0$, that is, $Z$ is of $\epsilon$-Fano type.

Next we assume that  $f$ is small and  $Z$ is of $\epsilon$-Fano type. Let $\Gamma$ be an effective big $\mathbb{Q}$-divisor on $Z$ such that $(Z, \Gamma)$ is $\epsilon$-klt and $K_Z+\Gamma\equiv 0$. Let $\Delta$ be the strict transform of $\Gamma$ on $Y$. Then $\Delta$ is big since $f$ is small. Again by $f$ is small, 
$$
K_Y+\Delta=f^*(K_Z+\Gamma).
$$
Hence $(Y, \Delta)$ is $\epsilon$-klt and $K_Y+\Delta\equiv 0$. Hence $Y$ is of $\epsilon$-Fano type. 
\end{proof}
\begin{proof}[Proof of Theorem \ref{thm MFS}]
Fix $0<\epsilon<1$, an integer $n>0$. Let  $X$ be a variety of $\epsilon$-Fano type of dimension $n$, that is, by Lemma \ref{FT=CY}, there exists an effective $\mathbb{Q}$-divisor  $\Delta$ such that $\Delta$ is big, $(X, \Delta)$ is $\epsilon$-klt, and $K_X+\Delta\equiv 0$. 
By \cite[Corollary 1.4.3]{BCHM},  taking $\bQ$-factorialization of $(X, \Delta)$, we have a birational morphism $\phi: X_0 \rightarrow X$ where
$K_{X_0}+\phi^{-1}_*\Delta=\phi^*(K_X+\Delta)$, $X_0$ is $\bQ$-factorial, and $\phi$ is isomorphic in codimension one. 

Again by \cite[Corollary 1.4.3]{BCHM},  taking terminalization of $X_0$, 
we have a birational morphism $\pi: X_1 \rightarrow X_0$ where
$K_{X_1}+\Delta_{X_1}=\pi^*(K_{X_0}+\phi^{-1}_*\Delta)$, 
$\Delta_{X_1}\geq \pi^*(\phi^{-1}_*\Delta)$ is an effective $\bQ$-divisor, $X_1$ is $\bQ$-factorial terminal.
Here $K_{X_1}+\Delta_{X_1}\equiv 0$ and $(X_1,\Delta_{X_1})$ is  $\epsilon$-klt. Since $\Delta$ is big and $\phi$ is small, $\Delta_{X_1}\geq \pi^*(\phi^{-1}_*\Delta)$ is big. Therefore, $X_1$ is $\bQ$-factorial terminal and of $\epsilon$-Fano type.

Running $K$-MMP on $X_1$, we get a sequence of normal projective varieties: 
$$
X_1\dashrightarrow X_2 \dashrightarrow X_3\dashrightarrow \cdots\dashrightarrow X_r\rightarrow T. 
$$ 
Since $-K_{X_1}$ is big, this sequence ends up with a Mori fiber space $X_r\rightarrow T$ (cf. \cite[Corollary 1.3.3]{BCHM}). Since we run $K$-MMP, $X_r$ is again $\bQ$-factorial terminal. 
By Lemma \ref{mmp lem}, for all $i$, $X_i$ is of $\epsilon$-Fano type. Now $X_r$ is an $n$-dimensional variety of  $\epsilon$-Fano type with a Mori fiber structure by construction, which is birational to $X$. 
We complete the proof of Theorem \ref{thm MFS}. 
\end{proof}

\section{Proof of Theorem \ref{NV}}\label{section proof 2}
In this section, for the reader's convenience, we recall the proof of Theorem \ref{NV} from \cite{Jiang6}. The proof of Theorem \ref{NV} follows from Lemmas \ref{lemma ample} and \ref{lem volume} below.
\subsection{Setting}\label{setting}Fix $0<\epsilon<1$. Let $X$ be a $3$-fold of $\epsilon$-Fano type with a Mori fibration $f:X\to Z$ and $\dim Z>0$. Suppose $(X, B)$  is an  $\epsilon$-klt log Fano pair.  We will explain more about the surface $G$ defined in Definition \ref{GF}.

If $\dim Z=1$, then $Z=\bP^1$. In this case $G$ is defined to be a general fiber of $f$, which is a smooth del Pezzo surface.

If $\dim Z=2$, then we want to explain the choice of $H$ first. We first claim that such $Z$ forms a bounded family. Since $X$ is of $\epsilon$-Fano type, there exist effective $\bQ$-divisors $\Delta$ and $\Delta'$ such that $(Z, \Delta)$ is klt, $-(K_Z+\Delta)$ is ample by \cite[Corollary 3.3]{FG}, and $(Z, \Delta')$ is $\delta$-klt, $-(K_Z+\Delta')\sim_\bQ 0$ by \cite[Corollary 1.7]{Birkar}. Note that $\delta$ is a positive number depends only on $\epsilon$. We may choose sufficiently small $t>0$ such that $(Z, (1-t)\Delta'+t\Delta)$ is still $\delta$-klt. In this case, 
$$
-(K_Z+(1-t)\Delta'+t\Delta)\sim -t(K_Z+\Delta)
$$
is ample. Hence $Z$ is of $\delta$-Fano type. Hence by BAB Conjecture in dimension $2$, such $Z$ forms a bounded family. This means that there is a positive integer $d_\epsilon$ depending only on $\epsilon$, and we may find a general very ample divisor $H$ on $Z$ such that $H^2\leq d_\epsilon$. Now we take $G=f^*H$, which is a conic bundle over the curve $H$ (i.e. $-K_G$ is ample over $H$). Note that $H$ and $G$ are smooth since $H$ is general. Also $(G, B|_G)$ is $\epsilon$-klt and $-(K_G+B|_G)+G|_G$ is ample by adjunction. Note that $G|_G=f^*(H|_H)$ is the sum of $(H^2)$ fibers of $f$ and $(H^2)\leq d_\epsilon$. Hence $-(K_G+B|_G)+d_\epsilon F$ is ample, where $F$ is a general fiber of $f$.

\subsection{Two boundedness theorems on surfaces}We recall two boundedness theorems on surfaces, the idea of proofs of them are from the proof of BAB Conjecture in dimension two by Alexeev--Mori \cite{AM}. 
\begin{thm}[{\cite[Theorem 2.8]{Jiang}}]\label{gac}
Fix $0<\epsilon<1$. 
Then there exists a number $\mu(2,\epsilon)>0$ depending only on $\epsilon$ with the following property:

 If $(X, B)$  is an  $\epsilon$-klt log Fano pair and $X$ is a smooth surface, then
$\alpha(X,B)\geq \mu(2,\epsilon).$
\end{thm}

\begin{thm}[{\cite[Theorem 1.7]{Jiang6}}]\label{conj LCT}
Fix $0<\epsilon<1$. Then there exists a number  $\lambda(2,\epsilon)>0$ depending only on $\epsilon$, satisfying the following property:
\begin{enumerate}
\item If $(G, B)$  is an  $\epsilon$-klt log Fano pair and $G$ is a smooth del Pezzo surface, then $(G, (1+t)B)$ is klt for $0<t\leq \lambda(2,\epsilon)$;

\item
If $f:G\to H$ is a conic bundle from a smooth surface $G$ to a smooth curve, $(G, B)$ is an $\epsilon$-klt pair and $-(K_G+B)+d_\epsilon F$ is ample, then $(X, (1+t)B)$ is klt for $0<t\leq \lambda(2,\epsilon)$. Here $F$ is a general fiber of $f$ and $d_\epsilon$ is the number depending only on $\epsilon$ defined in Subsection \ref{setting}.
\end{enumerate}
\end{thm}
Note that in \cite{Jiang6}, such a conic bundle $G$ in (2) is called a $(2,1, d_\epsilon,\epsilon)$-Fano fibration.
\subsection{Boundedness of ampleness}
\begin{lem}[{cf. \cite[Lemma 3.2]{Jiang6}}]\label{lemma ample}
Keep the setting in Subsection \ref{setting},  then there exists a positive integer $N_\epsilon$ depending only on $\epsilon$ such that $-K_X+kG$ is ample for all $k\geq N_\epsilon$.
\end{lem}
\begin{proof} 
Let $X$ be a $3$-fold of $\epsilon$-Fano type with a Mori fibration $f:X\to Z$ and $\dim Z>0$. Suppose $(X, B)$  is an  $\epsilon$-klt log Fano pair. 

By our construction, in either case, $(G, B|_G)$ satisfies one of the two conditions in Theorem \ref{conj LCT}. Hence $(G, (1+\lambda)B|_G)$ is klt for $\lambda=\lambda(2,\epsilon)$.

Hence, in either case, every curve in $\Nklt(X, (1+\lambda)B)$ is contracted by $f$ by inversion of adjunction. In particular, every curve $C_0$ supported in $\Nklt(X, (1+\lambda)B)$ satisfies that $G\cdot C_0 =0$.

Now we consider an extremal ray $R$ of $\overline{NE}(X)$. Since $X$ is of Fano type, $R$ is always generated by a rational curve by Cone Theorem.

If $R$ is $(K_X+(1+\lambda)B)$-non-negative, recall that $-(K_X+B)$ is ample, then 
$$
-K_X\cdot R= -\Big(1+\frac{1}{\lambda}\Big)(K_X+B)\cdot R+\frac{1}{\lambda}(K_X+(1+\lambda)B))\cdot R> 0.
$$

If $R$ is $(K_X+(1+\lambda)B)$-negative and $G\cdot R=0$, then $-K_X\cdot R>0$ since $-K_X$ is ample over $Z$ and $R$ is contracted by $f$.

If $R$ is $(K_X+(1+\lambda)B)$-negative and $G\cdot R>0$, then every curve generating $R$ is not supported in $\Nklt(X, (1+\lambda)B)$. By Theorem \ref{ext ray2}, $R$ is generated by a rational curve $C$ such that
$$
(K_X+(1+\lambda)B)\cdot C \geq -6.
$$
On the other hand, $G\cdot C\geq 1$ since $G\cdot C>0$ and $G$ is Cartier.
Hence 
\begin{align*}
{}&\Big(-K_X+\frac{6}{\lambda}G\Big)\cdot C\\
={}& -\Big(1+\frac{1}{\lambda}\Big)(K_X+B)\cdot C+\frac{1}{\lambda}(K_X+(1+\lambda)B)\cdot C+\frac{6}{\lambda}G\cdot C> 0.
\end{align*}

In summary, 
$$
(-K_X+kG)\cdot R> 0
$$
holds for every extremal ray $R$ and for all $k\geq \frac{6}{\lambda}$ (recall that $G$ is nef). By Kleiman's Ampleness Criterion, 
$-K_X+kG$ is ample for all $k\geq \frac{6}{\lambda}$. We may take 
$$
N_\epsilon=\frac{6}{\lambda(2,\epsilon)}
$$
and complete the proof.
\end{proof}
\subsection{Boundedness of volumes}

\begin{lem}[{cf. \cite[Theorem 4.1]{Jiang6}}]\label{lem volume}Keep the setting in Subsection \ref{setting},  then there exists a positive number $V_\epsilon$ depending only on $\epsilon$  such that $(-K_X+N_\epsilon G)^3\leq V_\epsilon$.
\end{lem}

\begin{proof}
Let $X$ be a $3$-fold of $\epsilon$-Fano type with a Mori fibration $f:X\to Z$ and $\dim Z>0$. Suppose $(X, B)$  is an  $\epsilon$-klt log Fano pair. 

If $\dim Z=1$,  $G$ is a smooth del Pezzo surface and $(G, B|_G)$ is an  $\epsilon$-klt log Fano pair. Note that $\Vol(G, -K_G)=K_G^2\leq 9$.
Assume that for some $w>0$,
$$(-K_X+N_\epsilon G)^3>27(N_\epsilon+w).$$
It suffices to find an upper bound for $w$.
We may assume that $w>2$. 
By Lemma \ref{lemma volume},
\begin{align*}
{}&\Vol(X, -K_X-wG)\\
\geq{}& \Vol(X, -K_X+N_\epsilon G)-3(N_\epsilon+w)\Vol(G, -K_G)>0.
\end{align*}
Hence there exists an effective $\bQ$-divisor $B'\sim_\bQ-K_X-wG$.
For two general fibers $G_1$ and $G_2$, consider the pair $(X, (1-\frac{2}{w})B+\frac{2}{w}B'+G_1+G_2)$ where $2/w<1$. Note that 
$$-\Big(K_{X}+ \Big(1-\frac{2}{w}\Big)B+\frac{2}{w}B'+G_1+G_2\Big)\sim_\bQ-\Big(1-\frac{2}{w}\Big)(K_X+B)$$
is ample.
By Connectedness Lemma,  
$\Nklt(X, (1-\frac{2}{w})B+\frac{2}{w}B'+G_1+G_2)$ is connected. On the other hand, it contains $G_1\cup G_2$, hence contains a non-klt center dominating $Z$. By inversion of adjunction, 
$(G, (1-\frac{2}{w})B|_G+\frac{2}{w}B'|_G)$ is not klt for a general fiber $G$. On the other hand, $(G, B|_G)$ is an $\epsilon$-klt log Fano pair of dimension $2$, $G$ is a  del Pezzo surface, and $B'|_G-B|_G\sim_\bQ-(K_G+B|_G)$.
Hence by Theorem \ref{gac}, $$\frac{2}{w}\geq \glct(G, B|_G; B'|_G-B|_G) \geq \mu(2, \epsilon).$$ Hence $w\leq \frac{2}{\mu(2,\epsilon)}$. In this case, we may take 
$$
V_\epsilon=27\Big(N_\epsilon + \frac{2}{\mu(2,\epsilon)}\Big).
$$

Now assume that $\dim Z=2$. 
As constructed in Subsection \ref{setting}, $G\to H$ is a conic bundle from a smooth surface to a smooth curve. 

\begin{claim}\label{claim}
$\Vol(G, -K_X|_{G}+N_\epsilon G|_G)\leq 8+4(N_\epsilon+1)d_\epsilon
$.
\end{claim}
\begin{proof}[Proof of Claim \ref{claim}]
Note that $-K_X+N_\epsilon G$ is ample, so is $-K_X|_{G}+N_\epsilon G|_G$. Also note that $(H^2)\leq d_\epsilon$ and $G|_G= (H^2)F$ where $F\simeq \bP^1$ is a general fiber of $f$. Hence
\begin{align*}
\Vol(G, -K_X|_{G}+N_\epsilon G|_G){}&=(-K_X|_{G}+N_\epsilon G|_G)^2\\
{}&=(-K_{G}+(N_\epsilon+1) G|_G)^2\\
{}&=(K_G)^2-2(N_\epsilon+1)K_G\cdot (H^2)F\\
{}&\leq 8+4(N_\epsilon+1)d_\epsilon
\end{align*}
Here we use the fact that for the conic bundle $G$, $(K_G)^2\leq 8$.
\end{proof}
Assume that for some $w>0$,
$$(-K_X+N_\epsilon G)^3>3(N_\epsilon+w)(8+4(N_\epsilon+1)d_\epsilon
).$$
It suffices to find an upper bound for $w$.
We may assume that $w>3$.
By Lemma \ref{lemma volume} and Claim \ref{claim},
\begin{align*}
{}&\Vol(X, -K_X-wG)\\
\geq {}& \Vol(X, -K_X+N_\epsilon G)-3(N_\epsilon+w)\Vol(G, -K_X|_{G}+N_\epsilon G|_G)>0.
\end{align*}
Hence there exists an effective $\bQ$-divisor $B'\sim_\bQ-K_X- wG$.
For a general fiber $F$ of $X$ over $z\in Z$, there exists a number $\eta>0$ (cf. \cite[4.8]{SOP}) such that for any general $H'\in |H|$ containing $z$, 
$$\Nklt\Big(X, \Big(1-\frac{3}{w}\Big)B+\frac{3}{w}B'\Big)=\Nklt\Big(X, \Big(1-\frac{3}{w}\Big)B+\frac{3}{w}B'+\eta f^*(H')\Big).$$
We may take general $H_j\in |H|$ containing $z$ for $1\leq j\leq J$ with $J>\frac{2}{\eta}$  and take $G_1=\sum_{j=1}^J\frac{2}{J}f^*(H_j)$. Then $\mult_{F}G_1\geq 2$ and $G_1\sim_\bQ 2f^*(H)\sim_\bQ 2G$. In particular, $(X, G_1)$ is not klt at $F$ and by construction, in a neighborhood of $F$,
\begin{align*}
{}&\Nklt\Big(X, \Big(1-\frac{3}{w}\Big)B+\frac{3}{w}B'\Big)\cup F\\
={}&\Nklt\Big(X, \Big(1-\frac{3}{w}\Big)B+\frac{3}{w}B'+G_1\Big).
\end{align*}
Take a general element $G_2\in |f^*(H)|$ not containing $F$, consider the pair $(X, (1-\frac{3}{w})B+\frac{3}{w}B'+G_1+G_2)$ where $3/w<1$. 
Then
$$-\Big(K_{X}+ \Big(1-\frac{3}{w}\Big)B+\frac{3}{w}B'+G_1+G_2\Big)\sim_\bQ-\Big(1-\frac{3}{w}\Big)(K_X+B)$$
is ample.
Since 
$$F\cup G_2\subset \Nklt\Big(X, \Big(1-\frac{3}{w}\Big)B+\frac{3}{w}B'+G_1+G_2\Big),$$
 by Connectedness Lemma,   there is a curve $C$ contained in $\Nklt(X, (1-\frac{3}{w})B+\frac{3}{w}B'+G_1+G_2)$, intersecting $F$ and not contracted by $f$. Hence $C$ is  contained in $\Nklt(X, (1-\frac{3}{w})B+\frac{3}{w}B')$ by the construction of $G_1$ and generality of $G_2$. Since $C$ intersects $F$, so does $\Nklt(X, (1-\frac{3}{w})B+\frac{3}{w}B')$. 
 By inversion of adjunction, 
$(F, (1-\frac{3}{w})B|_F+\frac{3}{w}B'|_F)$ is not klt for a general fiber $F$. On the other hand, $(F, B|_F)$ is $\epsilon$-klt and $F\simeq \bP^1$. Hence  $\frac{3}{w}\geq \epsilon$ by comparing the coefficients of $\frac{3}{w}B'|_F$.
Hence $w\leq \frac{3}{\epsilon}$. Hence we may take
$$V_\epsilon=3\Big(N_\epsilon+\frac{3}{\epsilon}\Big)(8+4(N_\epsilon+1)d_\epsilon
)$$
by definition of $w$.
\end{proof}

\section{Proof of main theorems}\label{section proof main}

\subsection{Pseudo-effectivity of $c_2(\FF_X)$}
We recall a criterion of pseudo-effectivity of second Chern classes due to Miyaoka \cite{Miy87}.
\begin{definition}[cf. {\cite[Section 6]{Miy87}}] Let $X$ be an $n$-dimensional normal projective variety. A torsion free sheaf $\EE$ is called {\it generically semi-positive} (or  {\it generically nef}) if  one of the following equivalent conditions holds:
\begin{enumerate}
\item For every quotient torsion free sheaf $\EE\to \mathcal{L}$ and any ample divisors $H_i$, 
$$
c_1(\mathcal{L})\cdot H_1\cdot H_2\cdot \cdots \cdot H_{n-1}\geq 0.
$$
\item $\EE|_C$ is nef for a general
curve $C = D_1\cap \ldots \cap D_{n-1}$ for general $D_i \in |m_iH_i|$ and $m_i \gg 0$ and any ample divisors $H_i$.
\end{enumerate}
\end{definition}
The equivalence of these two definitions follows from the Mehta--Ramanathan theorem \cite{MR} (cf. \cite[Theorem 2.5]{Miy87}).
\begin{thm}[{\cite[Theorem 6.1]{Miy87}}]\label{c2 pseff}
Let $X$ be a normal projective variety which is smooth in codimension $2$. Let $\EE$ be a torsion free sheaf on $X$ such that 
\begin{enumerate}
\item $c_1(\EE)$ is a nef $\bQ$-Cartier divisor, and
\item $\EE$ is generically semi-positive.
\end{enumerate}
Then $c_2(\EE)$ is pseudo-effective.
\end{thm}

To check the generic semi-positivity of $\FF_X$, it suffices to check that of $\mT_X$, which is proved by the following theorem.
\begin{thm}[{\cite[Proof of 1.2 (1)]{KMMT}}]\label{T nef}
Let $(X, B)$ be a $\bQ$-factorial klt log Fano pair such that $X$ is smooth in codimension $2$. Then $\mT_X$ is generically semi-positive.
\end{thm}
This theorem is implicated by \cite[Proof of 1.2 (1)]{KMMT}, 
 combining a structure theorem for the cone of nef curves (replacing \cite[Theorem-Definition 2.2]{KMMT} by \cite[Corollary 1.3.5]{BCHM}) and deformation theory of rational curves (\cite[(1.3) Corollary]{KoMiMo}).

 \begin{proof}[Proof of Theorem \ref{F pseff}] 
Recall that $X$ is of Fano type and with $\bQ$-factorial  terminal singularities (terminal singularities implies smooth in codimension $2$). Since $\mT_X$ is generically semi-positive by Theorem \ref{T nef} and $G$ is nef,
$\FF_X=\mT_X\oplus \OO_X(\Ne G)$ is again generically semi-positive.
Since $c_1(\FF_X)=-K_X+\Ne G$ is ample, $c_2(\FF_X)$ is pseudo-effective by Theorem \ref{c2 pseff}. \end{proof}

\subsection{Upper bound of Gorenstein indices}

In this subsection, we prove Theorem \ref{bound of r}. We start from the estimate of $(-K_X)\cdot c_2(X)$.
\begin{proof}[Proof of Theorem \ref{bound of r}(1)]
Note that 
$$c_2(\FF_X)=c_2(\mT_X\oplus \OO_X(N_\epsilon G))=c_2(X)-K_X\cdot \Ne G.$$
By Theorem \ref{F pseff}, $c_2(\FF_X)$ is pseudo-effective, and thus
$$
(-K_X+N_\epsilon G)\cdot (c_2(X)-K_X\cdot \Ne G)\geq 0.
$$
Hence 
$$
(-K_X)\cdot c_2(X)\geq -(-K_X+N_\epsilon G)\cdot (-K_X)\cdot \Ne G-\Ne G\cdot c_2(X).
$$
It suffices to prove the following lemma.
\begin{lem}The following inequalities hold:
\begin{enumerate}
\item $(-K_X+N_\epsilon G)\cdot (-K_X)\cdot \Ne G\leq \Ve;$
\item $G\cdot c_2(X)\leq 12+4\Ve/\Ne$.
\end{enumerate}
\end{lem} 
\begin{proof}Recall that $-K_X$ is big,  $G$ is nef with $G^3=0$, and $-K_X+N_\epsilon G$ is ample with $(-K_X+N_\epsilon G)^3\leq \Ve.$

For statement (1),
\begin{align*}
{}&(-K_X+N_\epsilon G)\cdot (-K_X)\cdot \Ne G\\
\leq {}& (-K_X+N_\epsilon G)\cdot (-K_X+\Ne G)\cdot \Ne G\\
\leq {}& (-K_X+N_\epsilon G)^3\leq \Ve.
\end{align*}

Now we prove statement (2).

If $\dim Z=1$, then $G$ is a del Pezzo surface and $G\cdot c_2(X)=c_2(G)\leq 11$.

If $\dim Z=2$, by the exact sequence 
$$
0\to \mT_G\to \mT_X|_G\to \mathcal{N}_{G/X}\to 0,
$$
we have
\begin{align*}
{}&G\cdot c_2(X)\\
= {}& c_2(G)+c_1(G)\cdot c_1(\mathcal{N}_{G/X})\\
= {}& 12\chi(\OO_G)-K_G^2-K_G\cdot G|_G.
\end{align*}
Note that $G$ is a conic bundle over $H$, hence $\chi(\OO_X)=1-g(H)\leq 1$. On the other hand, 
\begin{align*}
{}&-K_G^2-K_G\cdot G|_G\\
= {}& -(K_X+G)^2\cdot G -(K_X+G)\cdot G^2\\
= {}& -K_X\cdot (K_X+3G)\cdot G\\
\leq {}&-K_X\cdot (\Ne+3)G\cdot G\\
= {}&(-K_X+\Ne G)\cdot (\Ne+3)G^2\\
\leq {}&(-K_X+\Ne G)\cdot \frac{\Ne+3}{N_\epsilon^2}(-K_X+\Ne G)^2\\
\leq {}&\frac{\Ne+3}{N_\epsilon^2} \Ve \leq  \frac{4\Ve}{N_\epsilon}.
\end{align*}
Hence $G\cdot c_2(X)\leq 12+{4\Ve}/{N_\epsilon}$.
\end{proof}
By this lemma, 
$$
(-K_X)\cdot c_2(X)\geq -(5\Ve +12\Ne).
$$
Hence Theorem \ref{bound of r}(1) is proved.
\end{proof}

By  Reid's Riemann--Roch formula, we can get the upper bound of Gorenstein indices. This method  highly depends on the classification of $3$-dimensional terminal singularities. 

\begin{proof}[Proof of Theorem \ref{bound of r}(2)]Recall that $X$ has at most terminal singularities.  Recall that every terminal singularity in dimension $3$ can be deformed to a collection of cyclic quotient terminal singularities. The {\it basket} of singularities of $X$
 is the set of all such cyclic quotient terminal singularities. We denote it by $(b, r)$ 
if a cyclic quotient terminal singularity is of type $\frac{1}{r}(1,-1,b)$ for integers $b$ and $r$. Hence a basket is a set (allowing multiplicities) of the form 
$
\{(b_i, r_i)|i\in I\}.
$
Note that since $r_i$ is the local index of singularities, $\lcm\{r_i\}K_X$ is a Cartier divisor.
By Reid's Riemann--Roch formula (cf. \cite[Lemmas 2.2, 2.3]{Ka=0} or \cite[(10.3)]{YPG}), we have
$$
\chi(\OO_X)=\frac{1}{24}(-K_X)\cdot c_2(X)+\frac{1}{24}\sum(r_i-\frac{1}{r_i}),
$$
where $r_i$ runs over the basket of singularities.
Note that $\chi(\OO_X)=1$ by Kawamata--Viehweg vanishing theorem since $X$ is of Fano type.
Hence by Theorem \ref{bound of r}(1),
$$
\sum(r_i-\frac{1}{r_i})\leq 24+M_\epsilon.
$$
In particular, $r_i\leq 24+M_\epsilon$. Hence $(24+M_\epsilon)!$ is divisible by $\lcm\{r_i\}$, which implies that $(24+M_\epsilon)!K_X$ is Cartier.
\end{proof}

\subsection{Proof of Theorem \ref{main}}\label{subsection proof main}

\begin{proof}[Proof of Theorem \ref{main}]
Let $X$ be a $3$-fold of $\epsilon$-Fano type with a Mori fibration $f:X\to Z$. If $\dim Z=0$, then $X$ is a $\bQ$-factorial terminal Fano $3$-fold with Picard number one, which is bounded by Kawamata \cite{K}. So we only need to consider the case when $\dim Z>0$. 

Keep the notation in Theorem \ref{NV}.
By Theorem \ref{bound of r}, $L:=r_\epsilon(-K_X+\Ne G)$ is a Cartier ample divisor. Recall that $X$ is of Fano type. By Koll\'ar's effective base point free theorem (cf. \cite[1.1 Theorem, 1.2 Lemma]{EBPF}), $720L$ is base point free and $4321L$ is very ample. On the other hand, $L^3\leq r_\epsilon^3\Ve$. Hence $X$ is a subvariety of projective spaces with bounded degree. Such $X$ forms a bounded family by the boundedness of Chow variety.
\end{proof}

  \section*{Acknowledgements}
The author would like to thank the referees for suggestions on the expression of this paper.

\end{document}